\documentclass[12pt]{article}

\usepackage{mathtools}
\usepackage{amssymb}
\usepackage{amsthm}
\usepackage{xcolor}
\usepackage{tikz}

\DeclareSymbolFont{bbold}{U}{bbold}{m}{n}
\DeclareSymbolFontAlphabet{\mathbbold}{bbold}

\newcommand{\N}{\mathbb{N}}
\newcommand{\Z}{\mathbb{Z}}
\newcommand{\Q}{\mathbb{Q}}
\newcommand{\R}{\mathbb{R}}
\newcommand{\C}{\mathbb{C}}

\newcommand{\1}{\mathbbold{1}}

\newcommand{\loc}{\mathrm{loc}}
\newcommand{\lu}{\mathrm{unif}}
\newcommand{\unif}{\mathrm{unif}}
\newcommand{\Mloc}{\mathcal{M}_{\loc}(\R)}
\newcommand{\M}{\mathcal{M}_{\loc,\unif}(\R)}
\renewcommand{\c}{{\rm c}}
\renewcommand{\r}{{\rm r}}
\renewcommand{\i}{{\rm i}}
\newcommand{\Dir}{{\rm D}}
\newcommand{\Neu}{{\rm N}}
\newcommand{\SL}{\mathrm{SL}}
\newcommand{\F}{\mathrm{F}}

\newcommand{\ckern}{c\mkern1mu}

\makeatletter
\newcommand{\uD}{\@ifnextchar1{\u@DN{\Dir}{\mu_1}\@gobble}{\u@DN{\Dir}{\mu}}}
\newcommand{\uN}{\@ifnextchar1{\u@DN{\Neu}{\mu_1}\@gobble}{\u@DN{\Neu}{\mu}}}
\newcommand{\u@DN}[2]{u_{#1\mkern-1mu,\mkern1mu#2}}
\makeatother

\DeclareMathOperator{\spt}{spt}
\DeclareMathOperator{\diam}{diam}

\renewcommand{\Re}{\operatorname{Re}}

\providecommand{\form}{\tau}
\providecommand{\scpr}[2]{\left( #1 \,\middle|\, #2 \right)}
\renewcommand{\sp}{\scpr}
\newcommand{\from}{\colon}
\newcommand\rfrac[2]{\tfrac{#1}{\raisebox{0.1em}{$\scriptstyle#2$}}}
\newcommand\bigstrut{\rule{0pt}{2.2ex}}
\newcommand\intstrut{\rule{0pt}{2.7ex}}
\let\eps\varepsilon
\let\phi\varphi
\let\le\leqslant
\let\leq\leqslant
\let\ge\geqslant
\let\geq\geqslant

\makeatletter
\def\@row#1,{#1\@ifnextchar;{\@gobble}{&\@row}}
\def\@matrix{%
    \expandafter\@row\my@arg,;%
    \@ifnextchar({\\ \get@in@paren{\@matrix}}{\after@matrix}%
    }
\def\matrixtype#1#2#3{%
    \ifmmode\def\after@matrix{\end{#2}\right#3}%
    \else\def\after@matrix{\end{#2}\right#3$}$\fi
    \left#1\begin{#2}\get@in@paren{\@matrix}%
    }
\def\@column#1,{#1\@ifnextchar;{\@gobble}{\\ \@column}}
\newcommand\vect{}
\def\svect(#1){\left(\begin{smallmatrix}\@column#1,;\end{smallmatrix}\right)}
\def\vect{\get@in@paren{\@vect}}
\def\@vect{\left(\begin{matrix}\expandafter\@column\my@arg,;\end{matrix}\right)}
\def\get@in@paren#1({\def\my@arg{}\def\my@rest{}\def\after@get{#1}\get@arg}
\let\e@a\expandafter
\def\get@arg#1){\e@a\kl@test\my@rest#1(;}
\def\kl@test#1(#2;{\e@a\def\e@a\my@arg\e@a{\my@arg#1}%
                   \ifx:#2:\let\my@exec\after@get
                   \else\let\my@exec\get@arg
                        \e@a\def\e@a\my@arg\e@a{\my@arg(}%
                        \def@rest#2;%
                   \fi\my@exec}
\def\def@rest#1(;{\def\my@rest{#1\kl@zu}}
\def\kl@zu{)}

\makeatletter
\newcommand\MyPairedDelimiter{%
  \@ifstar{\My@Paired@Delimiter{{}}}
          {\My@Paired@Delimiter{}}%
}
\newcommand\My@Paired@Delimiter[4]{%
  \newcommand#2{%
    \@ifstar{\start@PD{#1}{\delimitershortfall=-1sp}{#3}{#4}}
            {\start@PD{#1}{}{#3}{#4}}%
  }%
}
\newcommand\start@PD[5]{%
  #1\mathopen{\mathpalette\put@delim@helper{\put@delim{#2}{#3}{.}{#5}}}%
  #5%
  \mathclose{\mathpalette\put@delim@helper{\put@delim{#2}{.}{#4}{#5}}}%
}
\newcommand\put@delim@helper[2]{%
  \hbox{$\m@th\nulldelimiterspace=0pt #2#1$}%
}
\newcommand\put@delim[5]{%
  \setbox\z@\hbox{$\m@th#5{#4}$}%
  \setbox\tw@\null
  \ht\tw@\ht\z@ \dp\tw@\dp\z@
  #1#5%
  \left#2\box\tw@\right#3%
}

\makeatother
\MyPairedDelimiter*{\abs}{\lvert}{\rvert}
\MyPairedDelimiter*{\norm}{\lVert}{\rVert}
\MyPairedDelimiter{\set}{\{}{\}}

\newcommand\llim{
\mathchoice{\vcenter{\hbox{${\scriptstyle{-}}$}}}
{\vcenter{\hbox{$\scriptstyle{-}$}}}
{\vcenter{\hbox{$\scriptscriptstyle{-}$}}}
{\vcenter{\hbox{$\scriptscriptstyle{-}$}}}}
\newcommand\rlim{
\mathchoice{\vcenter{\hbox{${\scriptstyle{+}}$}}}
{\vcenter{\hbox{$\scriptstyle{+}$}}}
{\vcenter{\hbox{$\scriptscriptstyle{+}$}}}
{\vcenter{\hbox{$\scriptscriptstyle{+}$}}}}

\theoremstyle{plain} 
\newtheorem{theorem}{Theorem}[section]
\newtheorem{corollary}[theorem]{Corollary}
\newtheorem{lemma}[theorem]{Lemma}
\newtheorem{proposition}[theorem]{Proposition}
\theoremstyle{definition}
\newtheorem{example}[theorem]{Example}
\newtheorem*{definition}{Definition}
\newtheorem{remark}[theorem]{Remark}

\newcommand{\avoidbreak}{\postdisplaypenalty50}

\usepackage{enumitem}

\setenumerate[1]{nolistsep} 
\setenumerate[2]{nolistsep} 

\begin{document}

\medmuskip=4mu plus 2mu minus 3mu
\thickmuskip=5mu plus 3mu minus 1mu
\belowdisplayshortskip=9pt plus 3pt minus 5pt

\title{A weak Gordon type condition for absence of eigenvalues of one-dimensional Schr\"odinger operators}

\author{Christian Seifert and Hendrik Vogt}

\maketitle

\begin{abstract}
  We study one-dimensional Schr\"odinger operators with complex measures as potentials
  and present an improved criterion for absence of eigenvalues
  which involves a weak local periodicity condition.
  The criterion leads to sharp quantitative bounds on the eigenvalues.
  We apply our result to quasiperiodic measures as potentials.
  
  MSC2010: 34L15, 34L40, 81Q10, 81Q12
  
  Key words: Schr\"odinger operators, eigenvalue problem, quasiperiodic potential
\end{abstract}

\section{Introduction}

In this paper we study Schr\"odinger operators of the form
\[H_\mu = -\Delta + \mu\]
in $L_2(\R)$, where the potential $\mu$ is locally a complex Radon measure.
We are going to study absence of eigenvalues of $H_\mu$ under the condition
that $\mu$ can be approximated by periodic measures in a suitable sense.
Roughly speaking, we require that there are arbitrarily large periods $p>0$
such that the three ``pieces'' $\1_{[-p,0]}\mu$, $\1_{[0,p]}\mu$ and
$\1_{[p,2p]}\mu$ look very similar.

The above so-called \emph{Gordon condition} first appeared in \cite{Gordon1976}.
Gordon treated the case of bounded potentials and measured the distance of the
``pieces'' in the $L_\infty$-norm. In \cite{DamanikStolz2000} this was generalised
to uniformly locally integrable potentials and the $L_{1,\loc,\unif}$-norm.
Recently, a similar result was shown in \cite{Seifert2011} for real measures
as potentials, and the distance was measured in the total variation metric.
We will extend these results in two directions.

The most important new aspect is that we work with a weaker Wasserstein type metric
in the Gordon condition, thus allowing a wider range of applications.
For linear combinations of Dirac measures this means that the positions
of the Dirac deltas are allowed to vary, not only their coefficients as
in~\cite{Seifert2011}.


The second extension concerns sharp quantitative results.
We show that the operator $H_\mu$ has no eigenvalues $z$ with $|z| < E_\mu$,
where $E_\mu$ is a constant given by an explicit expression in terms of $\mu$.
Moreover, we give an example that $-E_\mu$ can occur as an eigenvalue.
We note that we can also deal with complex measures, thus obtaining non-selfadjoint operators.

Our result can be applied to quasiperiodic measures where the ratio of the periods
satisfies a strong Liouville condition.
Such potentials were also considered in \cite{DamanikStolz2000,Seifert2011};
however, absence of eigenvalues was only shown under some additional hypotheses.
Working with our weaker metric we can prove the result without any further assumptions.

The question of the absence of eigenvalues appears in the study of one-dimensional quasicrystals. 
Here, it is crucial to obtain a good criterion when working with a weak metric to measure the distances of potentials (i.e., measures).
The results in \cite[Chapters 4 and~6]{Seifert2012} suggest that
the topology of vague convergence of measures is appropriate
in the study of continuum quasicrystals.
Our Wasserstein type metric is precisely of this quality, see Lemma~\ref{lem:vague}.

\medskip

We now describe in more detail the Schr\"odinger operator $H_\mu$ we are going to study.
We say that
\[
  \mu \from \set{B\subseteq\R;\;B \text{ is a bounded Borel set}} \to \C
\]
is a \emph{local measure} if $\1_K\mu := \mu(\cdot\cap K)$ is a
complex Radon measure for any compact set $K\subseteq \R$.
Then there exist a (unique) nonnegative Radon measure $\nu$ on $\R$ and
a measurable function $\sigma\from\R\to \C$ such that $\abs{\sigma} = 1$ $\nu$-a.e.\ and
$\1_K\mu = \1_K\sigma\nu$ for all compact sets $K\subseteq \R$. The \emph{total variation} of $\mu$ is defined by $\abs{\mu}:=\nu$.
Let $\Mloc$ be the space of all local measures on $\R$.

A local measure $\mu\in \Mloc$ is called \emph{uniformly locally bounded} if
\[
  \norm{\mu}_\lu := \sup_{a\in\R} \abs{\mu}((a,a+1]) < \infty.
\]
Let $\M$ denote the space of all uniformly locally bounded local measures.
The space $\M$ naturally extends $L_{1,\loc,\unif}(\R)$ to measures.

Given $\mu\in\M$, we now define the operator $H_\mu$ in $L_2(\R)$ as the
maximal operator associated with $-\Delta+\mu$ (in the distributional sense).
More precisely,
\begin{align*}
  D(H_\mu) & := \set{u\in L_2(\R)\cap C(\R);\; -u'' + u\mu\in L_2(\R)},\\
  H_\mu u & := -u'' + u\mu,
\end{align*}
i.e., for $u,f\in L_2(\R)$ one has $u\in D(H_\mu)$, $H_\mu u = f$ if and only if
\[
  \int_\R f\phi = -\!\int_\R u\phi'' + \int_\R u\phi\,d\mu \qquad \bigl(\phi\in C_\c^\infty(\R)\bigr).
\]
The operator $H_\mu$ can also be defined via the form method \cite{KlassertLenzStollmann2011}
or along the lines of Sturm-Liouville differential operators \cite{BenAmorRemling2005};
see Remark~\ref{rem:sl+form} for details.
In Theorem~\ref{thm:3ops} we will show that the three approaches yield the same operator
(cf.\ \cite[Corollary 1.3.7]{Seifert2012} for the self-adjoint case).
This result may be of independent interest.

\smallskip

\textit{Notation.}
For $\mu\in\M$ and $s,t\in\R$ we write
\[
  \int_s^t \dots \, d\mu = \begin{cases}
                         \int_{(s,t]} \dots \, d\mu & \text{ if } t\ge s,\\
                        -\int_{(t,s]} \dots \, d\mu & \text{ if } t<s.
                       \end{cases}
\]

The paper is organized as follows.
In Section~2 we introduce the seminorms on the space of measures we will work with.
Estimates on solutions of the eigenvalue equation are provided in Section~3.
These will then be used to prove the equivalence of various possibilities to define the Schr\"odinger operator.
In Section~4 we estimate the difference of solutions to different potentials.
Our main theorem is stated and proved in Section~5. 
We also give an application to quasiperiodic potentials there. 
Finally, in Section~6 we construct an example to show that our results are sharp.

\section{Seminorms on measures}

In this section we define suitable seminorms on $\M$.
The Lipschitz continuous functions on $\R$ will be denoted by 
$W_{\!\infty}^1(\R)$.

\begin{definition}
For $\mu \in \M$ and a set $I \subseteq \R$ (which will usually be an interval) we define
\[
  \norm{\mu}_I := \sup\set{\Bigl|\int u\, d\mu\Bigr|;\;
  u\in W_{\!\infty}^1(\R),\: \spt u \subseteq I,\: \diam\spt u \leq 2,\: \norm{u'}_\infty \le 1 }.
\]
\end{definition}

Choosing supports of length at most $2$ yields nice values for ``typical'' measures.
\begin{example}
  (a) For the Lebesgue measure $\lambda$ on $\R$ we have $\norm{\lambda}_\R = 1$.

  (b) For the Dirac measure $\delta_t$ at $t\in\R$ we have $\norm{\delta_t}_\R = 1$.
Moreover, $\norm{\sum_{n\in\Z} \delta_n}_\R = 1$.

  (c) Let $\eps\in(0,1)$ and $\mu := \delta_0 - \delta_\eps + \delta_2 - \delta_{2+\eps}$.
Then $\|\mu\|_\R = \eps$, but there exists $u\in W_{\!\infty}^1(\R)$ with $\diam\spt u = 2+\eps$ and $\norm{u'}_\infty\le 1$ such that $\int\mkern-2mu u\, d\mu = 2\eps$. Thus, increasing the diameter of $\spt u$ in the definition of $\norm{\mu}_I$ only slightly can change the norm by a factor of $2$.
\end{example}

Next we explain what happens if one further increases $\diam\spt u$.

\begin{remark}\label{rem:norm_spt}
  Let $\mu\in\M$, $n\in\N$, and let $u\in W_{\!\infty}^1(\R)$ such that $\spt u \subseteq [-n,n]$,
  $\norm{u'}_\infty\leq 1$. Then
  \[\abs{\int u\, d\mu} \leq n^2\norm{\mu}_{[-n,n]}.\]
  Indeed, there exist $u_k\in W_{\!\infty}^1(\R)$ with
  $\diam\spt u_k \leq 2$\, ($k=1-n,\ldots,n-1$)
such that $u = \sum u_k$ and $\sum \norm{u_k'}_\infty \leq n^2$:
let
\[
  u_k(t) = \begin{cases}
           u(t) - u(k-1)\cdot(k-t) & \text{ if } k-1 \le t \le k, \\
           u(k)\cdot(k+1-t)        & \text{ if } k < t \le k+1
           \end{cases}
\]
for $k<0$,
\[
  u_0(t) = \begin{cases}
           u(t) - u(-1)\cdot(-t) & \text{ if } {-}1 \le t \le 0, \\
           u(t) - u(1)\cdot t    & \text{ if } 0 < t \le 1,
           \end{cases}
\]
and $u_k$ for $k>0$ defined similarly as for $k<0$.
Then $\norm{u_k'}_\infty \le n-|k|$ for all~$k$ since $|u(t)| \le n-|t|$ for all $|t|\le n$,
and $\sum_{k=1-n}^{n-1} (n-|k|) = n^2$.
\end{remark}

\begin{remark}\label{rem:norm-comp}
(a) It is easy to see that $\norm{\cdot}_\R$ defines a norm on $\M$ satisfying 
$\norm{\mu}_\R \le \norm{\mu}_\lu$:
for $u\in W_{\!\infty}^1(\R)$ with $\spt u \subseteq [a-1,a+1]$ and $\norm{u'}_\infty\leq 1$ (with some $a\in\R$) one has
\[
  |u(s)| \le \max\set{1-|s-a|,0} = \int_{a-1}^a \1_{(t,t+1]}(s)\,dt
  \qquad (s\in\R)
\]
and hence
\[
  \abs{\int u\, d\mu} \le \int_{a-1}^a \int \1_{(t,t+1]}\, d\abs{\mu} \,dt \le 
\norm{\mu}_\lu.
\]

(b) For $\mu\in\M$, the translates $\mu(\cdot+\eps)$ converge to $\mu$ in the
$\norm{\cdot}_\R$-norm as $\eps\to0$,
so $\norm{\cdot}_\R$ is weaker than $\norm{\cdot}_\lu$: for $\eps>0$ one has
\[
  \norm{\mu-\mu(\cdot+\eps)}_\R \le 3\eps\norm{\mu}_\lu.
\]
Indeed, for $\eps\ge1$ this is clear by part~(a), and for $0<\eps<1$ it follows from
the estimate
\[
  \abs{\int_\R u\,d\bigl( \mu-\mu(\cdot+\eps) \bigr)}
  \le \int_\R \abs{u-u(\cdot-\eps)}\,d\abs{\mu}
  \le \int_{\spt u + (0,\varepsilon)} \eps \norm{u'}_\infty \,d\abs{\mu}.
\]
\end{remark}

The next lemma shows in particular that the norm $\norm{\cdot}_\R$ induces the vague topology
on $\set*{\mu\in\M;\; \norm{\mu}_\lu\le R,\ \spt\mu \subseteq [-R,R]}$, for any $R>0$.

\begin{lemma}\label{lem:vague}
Let $(\mu_n)$ be a sequence in $\M$ and $\mu\in\M$. Then the following are equivalent:
  \begin{enumerate}[label=\textup{(\roman*)}]
    \item
      $\mu_n\to\mu$ vaguely,
    \item
      $(\1_I\mu_n)$ is $\norm{\cdot}_\lu$-bounded in $\M$ and $\norm{\mu_n-\mu}_I\to 0$,
      for all bounded intervals $I\subseteq\R$.
  \end{enumerate}
\end{lemma}

\begin{proof}
Without loss of generality assume that $\mu=0$. For a bounded open interval $I\subseteq\R$
we can regard $\set{\1_I\mu;\; \mu\in\M}$ as the dual space of $C_0(I)$.
Note that then $\mu_n\to0$ vaguely if and only if
\begin{enumerate}
  \item[(i')] $\1_I\mu_n\to0$ in the weak$^*$ topology, for all bounded open intervals $I\subseteq\R$.
\end{enumerate}
We now fix a bounded open interval $I\subseteq\R$.

(i') $\Rightarrow$ (ii):
As a null sequence in the weak$^*$ topology, $(\1_I\mu_n)$ is bounded and hence $\norm{\cdot}_\lu$-bounded;
moreover, $(\1_I\mu_n)$ tends to $0$ uniformly on compact subsets of $C_0(I)$.
By the Arzel\`a-Ascoli theorem,
\[
  \set{u\in W_{\!\infty}^1(\R);\; \spt u \subseteq I,\ \diam\spt u \leq 2,\ \norm{u'}_\infty \le 1}
  \avoidbreak
\]
is compact in $C_0(I)$, so we conclude that $\norm{\mu_n}_I\to 0$.

(ii) $\Rightarrow$ (i'):
It follows from Remark~\ref{rem:norm_spt} that $\int\mkern-2mu u\,d\mu_n \to 0$ for all Lipschitz functions
$u$ in $C_0(I)$. Since the Lipschitz functions are dense in $C_0(I)$ and $(\1_I\mu_n)$ is bounded
in $C_0(I)'$, this implies $\int\mkern-2mu u\,d\mu_n \to 0$ for all $u\in C_0(I)$.
\end{proof}

An important property of the norm $\norm{\cdot}_\R$ is that unlike $\norm{\cdot}_\lu$
it allows approximation of  measures in $\M$ by absolutely continuous measures
with $C^\infty$-densities.

\begin{proposition}
\label{prop:approx}
	Let $\mu\in\M$. Then there exists a sequence $(\mu_n)$ in $\M$ such
	that $\mu_n$ has a $C^\infty$-density and $\norm{\mu_n}_\lu\leq \norm{\mu}_\lu$ for all $n\in\N$,
        $\norm{\mu - \mu_n}_\R\to 0$,  
	and\/ $\limsup_{n\to\infty}\abs{\mu_n}(I) \leq \abs{\mu}(I)$
	for all compact intervals $I\subseteq\R$.
\end{proposition}

\begin{proof}
	For $n\in\N$ let $\psi_n \in C_\c^\infty(\R)$ satisfy
	$\spt \psi_n \subseteq(-\rfrac{1}{n},\rfrac{1}{n})$, $\int \psi_n = 1$, $\psi_n(-\cdot)=\psi_n$
	and $\psi_n\ge0$.
	In particular, $(\psi_n)$ is an approximation of the identity.
	For $n\in\N$ we define
	\[f_n(x):= \int \psi_n(x-y)\,d\mu(y) = \psi_n * \mu(x) \qquad(x\in\R);\]
	then $f_n\in C^\infty(\R)$. Set $\mu_n:=f_n\lambda$.
	
	Let $u\in W_{\!\infty}^1(\R)$, $\diam\spt u \leq 2$, $\norm{u'}_\infty\le 1$. Then 
	\[\abs{\int u\, d(\mu-\mu_n)} = \abs{\int \bigl(u(y) - \psi_n*u(y)\bigr)\, d\mu(y)}.\]
	Moreover,
	\[\diam\spt(\psi_n*u) \le 2+\frac{2}{n}, \qquad
          \norm{u - \psi_n*u}_\infty \le \frac{1}{n}\]
	since $\norm{u'}_\infty\le 1$ and $\spt\psi_n \subseteq (-\rfrac1n,\rfrac1n)$. 	
	Thus, $\norm{\mu - \mu_n}_\R\to 0$.

	For $a,b\in\R$, $a<b$ we compute, applying Fubini's theorem twice and substituting $x=t+y$,
	\begin{align*}
	  \abs{\mu_n}([a,b])
	  & = \int_a^b \abs{f_n(x)}\, dx \le \int_a^b \!\int \psi_n(x-y)\,d\abs{\mu}(y)\,dx \\
	  & = \int\! \int_{a-t}^{b-t}\,d\abs{\mu}(y)\, \psi_n(t)\,dt
            = \int \abs{\mu}\bigl((a-t,b-t]\bigr)\psi_n(t)\, dt.
	\end{align*}
	On the one hand, choosing $b=a+1$ yields
	\[\norm{\mu_n}_\lu = \sup_{a\in\R} \abs{\mu_n}([a,a+1]) \le \norm{\mu}_\lu.\]
	On the other hand we obtain
	\[
	  \limsup_{n\to\infty} \abs{\mu_n}([a,b])
	  \le \limsup_{n\to\infty} \int \abs{\mu}\bigl((a-\rfrac{1}{n},b+\rfrac{1}{n}]\bigr)\psi_n(t)\, dt
	    = \abs{\mu}([a,b]).
	  \qedhere
	\]
\end{proof}

\begin{lemma}
\label{lem:norm_psi_mu}
Let $\mu\in\M$ and $\psi\in W_{\!\infty}^1(\R)$, and let $I \subseteq \R$ be an interval. Then
\[
  \norm{\psi \mu}_I 
  \le (\norm{\psi}_\infty + \norm{\psi'}_\infty)\norm{\mu}_{I \mkern1mu\cap\mkern2mu \spt\psi}.
\]
\end{lemma}

\begin{proof}
	Let $u\in W_{\!\infty}^1(\R)$ satisfy $\spt u\subseteq I$, $\diam\spt u \le 2$ and $\norm{u'}_\infty \leq 1$.
	Then $\psi u\in W_{\!\infty}^1(\R)$, $\spt (\psi u)\subseteq I\cap\spt\psi$, $\diam\spt (\psi u) \le 2$ and
	\[\norm{(\psi u)'}_\infty \le \norm{\psi}_\infty + \norm{\psi'}_\infty.\]
	Hence, the assertion follows.
\end{proof}

For $\mu\in\M$ we define $\phi_\mu\from\R\to\C$ by
\[
  \phi_\mu(t) := \int_0^t d\mu
  = \begin{cases}
     \mu\bigl((0,t]\bigr) & \text{ if } t\ge0, \\
    -\mu\bigl((t,0]\bigr) & \text{ if } t<0.
    \end{cases}
\]
The following result provides an alternative way for computing the $\norm{\cdot}_I$-seminorm of a given measure.

\begin{proposition}\label{prop:charakterisierung}
  Let $\mu\in\M$ and $a\in\R$.
  If $\mu$ is real, then
  \[\norm{\mu}_{[a-1,a+1]} = \min_{c\in\R} \int_{a-1}^{a+1} \abs{\phi_\mu(t)-c}\, dt.\]
  If $\mu$ is complex, then
  \[\norm{\mu}_{[a-1,a+1]} \le \min_{c\in\C} \int_{a-1}^{a+1} \abs{\phi_\mu(t)-c}\, dt \le 2\norm{\mu}_{[a-1,a+1]}.\]
  (Note that the minimum exists since the integral depends continuously on $c$
  and tends to $\infty$ as $|c|\to\infty$.)
\end{proposition}

\begin{proof}
  Let $u\in W_{\!\infty}^1(\R)$ with $\spt u \subseteq [a-1,a+1]$ and $\norm{u'}_\infty\leq 1$. Then
  \begin{align*}
    \int u\, d\mu & = \int_{a-1}^{a+1} \left( -\int_t^{a+1} u'(s)\, ds \right) d\mu(t) \\
    & = -\int_{a-1}^{a+1} \int_{a-1}^s \,d\mu(t)\, u'(s)\, ds 
      = -\int_{a-1}^{a+1} u'(s) \mu((a-1,s])\, ds.
  \end{align*}
  Since $\int_{a-1}^{a+1} u'(s)\, ds = 0$, it follows that
  \begin{align*}
    \int u\, d\mu & = -\int_{a-1}^{a+1} u'(s) (\phi_\mu(s)-c)\, ds
  \end{align*}
  for all $c\in\C$. Thus,
  \[\abs{\int u\, d\mu} \leq \int_{a-1}^{a+1} \abs{\phi_\mu(s)-c}\, ds\]
  and hence
  \[\norm{\mu}_{[a-1,a+1]} \leq \min_{c\in\C} \int_{a-1}^{a+1} \abs{\phi_\mu(s)-c}\, ds.\]

  In case $\mu$ is real, we choose $c\in\R$ to be a ``median'' of $\phi_\mu$,
  i.e., both $\phi_\mu \ge c$ and $\phi_\mu \le c$ hold on subsets of $[a-1,a+1]$
  with Lebesgue measure at least $1$.
  Then we can construct a real-valued function $u\in W_{\!\infty}^1(\R)$
  with $\spt u \subseteq [a-1,a+1]$ and $\norm{u'}_\infty\leq 1$ such that
  \[\int_{a-1}^{a+1} \abs{\phi_\mu(s)-c}\, ds = -\int_{a-1}^{a+1} u'(s) (\phi_\mu(s)-c)\, ds = \int u\, d\mu.\]
  Therefore,
  \[\min_{c\in\R} \int_{a-1}^{a+1} \abs{\phi_\mu(s)-c}\, ds \leq \norm{\mu}_{[a-1,a+1]}.\]

  In case $\mu$ is complex, we have $\mu = \mu_\r + i\mu_\i$ with real measures $\mu_\r,\mu_\i\in \M$. Then by the above we find $c_\r\in\R$ and a real-valued function $u\in W_{\!\infty}^1(\R)$ such that
  \[
    \int_{a-1}^{a+1} \abs{\phi_{\mu_\r}(s)-c_\r}\, ds
    = \int u\,d\mu_r = \Re \int u\,d\mu \le \norm{\mu}_{[a-1,a+1]}.
  \]
  In the same way we obtain $\int_{a-1}^{a+1} \abs{\phi_{\mu_\i}(s)-c_\i}\, ds
  \le \norm{\mu}_{[a-1,a+1]}$ for some $c_\i\in\R$.
  Since $\phi_\mu = \phi_{\mu_\r} + i\phi_{\mu_\i}$, it follows that
  \[
    \int_{a-1}^{a+1} \abs*{\phi_{\mu}(s)-(c_\r+ic_\i)}\, ds \leq 2\norm{\mu}_{[a-1,a+1]}.
    \qedhere
  \]
\end{proof}

\begin{remark}
\label{rem:c_0} 
(a) According to Proposition~\ref{prop:charakterisierung}, for any $\mu\in\M$ and $a\in\R$
    there exists $c_a=c_{\mu,a}\in\C$ such that
    \begin{equation}\label{eq:normalternative}
      \int_{a-1}^{a+1} \abs{\phi_\mu(t)-c_a}\, dt \leq 2\norm{\mu}_{[a-1,a+1]}.
    \end{equation}
    Moreover, $c_{\mu,0}$ can always be chosen such that
      \[\abs{c_{\mu,0}} \le \norm{\mu}_\lu\]
    since $\abs{\phi_\mu(t)}\le \norm{\mu}_\lu$ for $t\in[-1,1]$.

(b) We point out that $\abs{c_{\mu,0}}$ can be large even if $\norm{\mu}_\R$ is small.
    As an example consider $\mu = \delta_{\eps}-\delta_{-\eps}$ for small $\eps>0$,
    where one has $\norm{\mu}_\R = 2\eps$ and $\phi_\mu = \1_{\R\setminus[-\eps,\eps)}$.
\end{remark}

\begin{lemma}\label{lem:est_1}
Let $\mu\in\M$ and $\alpha,\beta\in\Z$, $\alpha\le-1$, $\beta\ge1$. Then
\[
  \int_k^{k+1} \abs{\phi_\mu(t)-c_0}\, dt \le 
2\max\set{k+1,-k}\norm{\mu}_{[\alpha,\beta]}
\]
for all $k\in[\alpha,\beta-1]$, with $c_0$ as in~\eqref{eq:normalternative}.
\end{lemma}

\begin{proof}
For $k=-1$ and $k=0$ the assertion follows from~\eqref{eq:normalternative}.
We now show the assertion for $k\ge1$; the proof of the case $k\le-2$ is analogous.

Using a telescope sum expansion, we estimate
\begingroup
\renewcommand\-{\mkern-1mu - \mkern-1mu}
\begin{align*}
\abs{c_k \- c_0}
 &\le \sum_{j=1}^k \int_{j-1}^j \bigl( \abs{c_j \- \phi_\mu(t)} + \abs{\phi_\mu(t) \- c_{j-1}} \bigr)\,dt \\
 &= \int_0^1 \abs{\phi_\mu(t) \- c_0}\,dt + \sum_{j=1}^{k-1} \int_{j-1}^{j+1} \abs{\phi_\mu(t) \- c_j}\,dt
      + \int_{k-1}^k \abs{\phi_\mu(t) \- c_k}\,dt.
\end{align*}
\endgroup
By~\eqref{eq:normalternative} we conclude that
\begin{align*}
\int_k^{k+1} \abs*{\phi_\mu(t)-c_0}\,dt
 &\le \int_k^{k+1} \bigl(\abs{\phi_\mu(t)-c_k}+\abs{c_k-c_0}\bigr)\,dt \\
 &\le \sum_{j=0}^k \int_{j-1}^{j+1} \abs{\phi_\mu(t)-c_j}\,dt
  \le (k+1) \cdot 2\norm{\mu}_{[\alpha,\beta]}.
  \qedhere
\end{align*}
\end{proof}

\section{Estimates on solutions}
\label{sec:Estimates}

In this section we provide estimates on functions that satisfy
the eigenvalue equation for $H_\mu$ (but that are not necessarily in $L_2$).
These estimates enable us to show that the different versions for defining $-\Delta+\mu$ actually lead to the same operator.

\begin{definition}
For $\mu\in \M$ and $z\in \C$ we say that $u\in L_{1,\loc}(\R)$ is a (generalized) solution of the equation $H_\mu u = zu$ if
$u\in C(\R)$ and
\[-u''+\mu u = zu\]
in the sense of distributions.
\end{definition}

\begin{remark}\label{rem:ftc}
Let $\mu\in\M$ and $z\in\C$.

(a) Note that $u$ is a solution of $H_\mu u = zu$ if and only if $u$ is a solution of $H_{\mu-z\lambda} u = 0$.
Therefore, it is no loss of generality to assume $z=0$ in this and the following section.

(b) Let $u$ be a solution of the equation $H_\mu u = 0$.
Then $u''$ is a local measure, so $u'$ has locally bounded variation
(and therefore its one-sided limits exist everywhere), and
\[
  u'(t\rlim) - u'(s\rlim) = \int_s^t u(r)\, d\mu(r)
\]
for all $s,t\in\R$.
In particular, any solution of $H_\mu u = 0$ is locally Lipschitz continuous.
\end{remark}

\begin{lemma}\label{lem:expo}
Let $\mu\in\M$, and let $u$ be a solution of the equation $H_\mu u=0$. Then
\[
  \abs{u(t)} + \abs{u'(t\rlim)}
  \le \bigl( \abs{u(0)} + \abs{u'(0\rlim)} \bigr) e^{\omega(|t|+1)} \qquad (t\in\R),
\]
with $\omega = \norm{\mu}_\lu + 1$.
\end{lemma}

\begin{proof}
Below we will use Gronwall's inequality in the following form: let 
$\phi\from[0,\infty)\to[0,\infty)$ be Borel measurable and locally bounded, 
$\eps>0$ and 
$0\leq \nu \in \mathcal{M}_\loc([0,\infty))$ (a nonnegative Radon measure on $[0,\infty)$), with
\[
  \phi(t) \le \eps + \int_{[0,t)} \phi(s)\,d\nu(s) \qquad (t\ge0).
  \avoidbreak
\]
Then $\phi(t) \le \eps e^{\nu([0,t))}$ for all $t\ge0$; see \cite[Thm.~5.1,\ p.\,498]{EthierKurtz2005}.

By Remark~\ref{rem:ftc}(b) we have
\[
  \vect(u(t),u'(t\rlim))
  = \vect(u(0),u'(0\rlim)) + \vect(\int_0^t u'(s\rlim)\,ds, \int_0^t u(s)\,d\mu \intstrut)
\]
for all $t\in\R$.
Thus, for $\phi(t) := \abs{u(t)} + \abs{u'(t\rlim)}$ and $\nu := \lambda+|\mu|$ we obtain
\[
  \phi(t) \le \phi(0) + \int_{(t,0]} \phi(s)\,d\nu(s) \qquad (t<0).
\]
By Gronwall's inequality we infer that
\[
  \phi(t) \le \phi(0) e^{\nu((t,0])} \qquad (t\le0).
\]
Thus, for $t\le0$ the assertion follows since $\norm{\nu}_\lu \le 1 + \norm{\mu}_\lu = \omega$ 
and hence $\nu((t,0]) \le \omega(\abs{t}+1)$.

Now let $t>0$. As above we can estimate
\[
  \phi_-(s) := \abs{u(s)} + \abs{u'(s\llim)} \le \phi(0) + \int_{(0,s)} \phi_-(r)\,d\nu(r),
\]
and Gronwall's inequality yields
\[
  \abs{u(s)} + \abs{u'(s\llim)}
  \le \bigl( \abs{u(0)} + \abs{u'(0\rlim)} \bigr) e^{\nu((0,s))} \qquad (s>0).
\]
For $s\downarrow t$ the assertion follows since $\nu((0,t]) \le \omega(\abs{t}+1)$.
\end{proof}

\begin{remark}
One could refine the above proof to obtain an estimate with $\omega = 2\norm{\mu}_\lu^{1/2}$,
at the expense of a multiplicative constant in the right hand side.
By a bootstrap-like argument we will even improve the bound to $\omega = \norm{\mu}_\lu^{1/2}$ in Proposition~\ref{prop:exp_growth}.
\end{remark}

In the following elementary lemma we provide estimates on the derivative
of a solution in terms of norms of the solution itself.

\begin{lemma}
\label{lem:u'_leq_u}
Let $\mu\in\M$, and let $u\in C(\R)$ be a solution of $H_\mu u = 0$.
Then for any interval $I = [a,b] \subseteq \R$ of length $1$ one has
\[
  \norm{u'|_I}_2
  \le \norm{u'|_I}_\infty
  \le M_\mu \norm{u|_I}_\infty
  \le \sqrt3\,M_{\mkern-0.5mu\mu}^{\mkern1mu 3/2} \norm{u|_I}_2,
\]
where $M_\mu := \norm{\mu}_\lu + 2$.
In particular, $u\in L_2(\R)$ implies $u'\in L_2(\R)$,
and $u\in C_0(\R)$ implies that $u'(t\rlim)\to 0$ as $t\to \pm\infty$.
\end{lemma}

\begin{proof}
The first inequality is clear.
We have
\[
  \abs{\int_a^b u'(t+)\, dt}
  = \abs*{u(b)-u(a)}
  \le 2\norm{u|_I}_\infty,
\]
so there exists $t_0\in I$ with $\abs{u'(t_0\rlim)} \le 2\norm{u|_I}_\infty$.
For $t\in I$ it follows with Remark~\ref{rem:ftc}(b) that
\[
  \abs{u'(t\rlim)}
  \le \abs{u'(t_0\rlim)} + \abs{\int_{t_0}^t u(s)\, d\mu(s)}
  \le (2+\norm{\mu}_\lu) \norm{u|_I}_\infty,
\]
which proves the second inequality.

Let now $s_0 \in I$ such that $\abs{u(s_0)} = \norm{u|_I}_\infty$.
Then by the above we have $\abs{u(s_0+t)} \ge (1-M_\mu|t|) \norm{u|_I}_\infty$
for all $t\in\R$ with $s_0+t\in I$. Therefore,
\[
  \norm{u|_I}_2^2
  \ge \int_0^{1/M_\mu} (M_\mu t)^2 \norm{u|_I}_\infty^2\,dt
    = \frac{1}{3M_\mu} \norm{u|_I}_\infty^2,
\]
and the last inequality follows.
\end{proof}

\begin{remark}\label{rem:sl+form}
  As mentioned in the introduction, different methods for defining $-\Delta+\mu$ appear in the literature. 

  (a)
  In \cite{BenAmorRemling2005} the operator is defined along the lines of Sturm-Liouville differential operators. 
  For $u\in W_{1,\loc}^1(\R)$ one defines $A_\mu u\in L_{1,\loc}(\R)$ by
  \[(A_\mu u)(t) := u'(t) - \int_0^t u(s)\, d\mu(s)\]
  for a.a.~$t\in\R$.
  (Here and in the following, we choose the continuous representative of $u$ in the integral.)

  Then we define $H_{\mu,\SL}$ as a realization of $-\Delta+\mu$ by
  \begin{align*}
    D(H_{\mu,\SL}) & := \set{u\in L_2(\R)\cap W_{1,\loc}^1(\R);\;
                             A_\mu u \in W_{1,\loc}^1(\R),\ (A_\mu u)'\in L_2(\R)}, \\
    H_{\mu,\SL} u  & := -(A_\mu u)'.
  \end{align*}
  
%
  
  (b)
  In \cite{KlassertLenzStollmann2011} the form method is used to define the operator.
  It is well-known that $\mu$ defines an infinitesimally form small perturbation
  of the classical Dirichlet form $\form_0$ given by
  \[
    D(\form_0) := W_2^1(\R), \qquad
    \form_0(u,v) := \int u' \overline{v}',
  \]
  i.e., for all $\delta>0$ there exists $C_\delta>0$ such that
  \[
    \abs{\int \abs{u}^2\,d\mu} \leq \delta \form_0(u,u) + C_\delta \norm{u}_2^2
    \qquad (u\in D(\form_0)).
  \]
  Thus,  
  \[
    D(\form_\mu) := W_2^1(\R), \qquad
    \form_\mu(u,v) := \int u' \overline{v}' + \int u\overline{v}\, d\mu
  \]
  defines a closed sectorial form $\form_\mu$ in $L_2(\R)$.
  (For the second integral note that $u\overline{v}\mu$ defines a complex Radon measure
  for any $u,v\in W_2^1(\R)$.)

  We denote by $H_{\mu,\F}$ the m-sectorial operator associated with $\form_\mu$, i.e.,
  \begin{align*}
  	D(H_{\mu,\F}) &= \set*{ u\in D(\form_\mu);\; \exists\mkern1mu f_u \in L_2(\R)\;
	\forall\mkern1mu v\in D(\form_\mu) \colon \form_\mu(u,v) = \sp{f_u}{v} }, \\
	H_{\mu,\F}\mkern1mu u  &= f_u.
  \end{align*}
\end{remark}

\begin{theorem}\label{thm:3ops}
With the above notation one has $H_\mu = H_{\mu,\SL} = H_{\mu,\F}$.
\end{theorem}
  
\begin{proof}
  It is easy to see that $H_{\mu,\F}\subseteq H_\mu$; we show the converse operator inclusion.
  Since $H_{\mu,\F}$ is sectorial, there exists $z\in\C$ such that $z+H_{\mu,\F}$
  is injective and surjective. Let $u$ be in the kernel of $z+H_\mu$;
  then $u$ is a solution of $H_{\mu+z\lambda} u = 0$.
  Since $u\in L_2(\R)$, Lemma~\ref{lem:u'_leq_u} yields $u'\in L_2(\R)$.
  For $\phi\in C_\c^\infty(\R)$ we obtain
  \[
    \sp{H_{\mu}u}{\phi} = -\sp{u}{\phi''} + \int u\overline{\phi}\,d\mu
    = \sp{u'}{\phi'} + \int u \overline{\phi}\,d\mu = \form_\mu(u,\phi).
  \]
  Since $C_\c^\infty(\R)$ is dense in $W_2^1(\R)$, it follows that
$u\in D(H_{\mu,\F})$ and $H_{\mu,\F} u = H_\mu u$. Hence also 
$(z+H_{\mu,\F})u = 0$. Thus $u=0$, so we have shown that $z+H_\mu$ is injective.
  Since surjective mappings cannot have proper injective extensions, we obtain 
$H_{\mu,\F}= H_\mu$. 
  
  Applying Fubini's theorem as in \cite[Lemma 1.3]{Seifert2011}, we obtain 
$H_{\mu,\F}\subseteq H_{\mu,\SL}$. For $u\in D(H_{\mu,\SL})$ and $\phi\in 
C_\c^\infty(\R)$ we compute, with the help of Fubini's theorem,
  \[\sp{H_{\mu,\SL} u}{\phi} = \int u'\overline{\phi}' - \int\! \int_0^t u(s)\,d\mu(s)\,\overline{\phi}'(t)\, dt = -\sp{u}{\phi''} + \int u\overline{\phi}\,d\mu.\]
  Thus, $u\in D(H_\mu)$ and $H_\mu u = H_{\mu,\SL} u$. Hence, we obtain 
$H_{\mu,\SL}\subseteq H_\mu = H_{\mu,\F}$ and therefore the equality of all 
three operators.
\end{proof}

\section{Estimates on the difference of two solutions}

Throughout this section let $\mu_1,\mu_2\in\M$, and let $u_1,u_2$ be solutions to
\[
  H_{\mu_1} u_1 = 0 \qquad \text{and} \qquad H_{\mu_2} u_2 = 0,
\]
respectively. We set
\[
  v:=u_1-u_2 \qquad \text{and} \qquad \nu:=\mu_1-\mu_2,
\]
and we aim for an estimate on $v$ on an interval $I$ in terms of 
$\norm{\nu}_I$.

\smallskip

We first recall the definition and properties of the transfer matrices associated with a measure $\mu$.

\begin{remark}
Let $\mu\in\M$, and let $u$ be a solution of the equation $H_\mu u = 0$.
Fix $s\in\R$. By Lemma~\ref{lem:expo}, $u$ is uniquely determined by $(u(s),u'(s\rlim))$
(see also \cite[Theorem 2.3]{BenAmorRemling2005}).
Thus, for $t\in \R$ we can define the transfer matrix mapping the solution of $H_\mu u = 0$ at $s$ to the solution at $t$, i.e.,
\[
  T_\mu(t,s)\from \begin{pmatrix} u(s)\\u'(s\rlim)\end{pmatrix} \mapsto \begin{pmatrix} u(t)\\u'(t\rlim)\end{pmatrix}.
\]
Note that $T_\mu(t,t) = (\!\begin{smallmatrix}1&0\\0&1\end{smallmatrix}\!)$
and $T_\mu(t,r) = T_\mu(t,s)T_\mu(s,r)$ for all $r,s,t\in\R$.

Let $\uN(\cdot,s),\uD(\cdot,s)$ be the solutions of $H_\mu u = 0$ satisfying Neumann and Dirichlet boundary conditions at $s$, respectively, i.e.,
\begin{alignat*}{2}
                  \uN(s,s) & = 1,        &                 u_D(s,s) & = 0, \\
  \partial_1 \uN(s\rlim,s) & = 0, \qquad & \partial_1 \uD(s\rlim,s) & = 1.
\end{alignat*}
Then
\[
  T_\mu(t,s) = \begin{pmatrix} \uN(t,s) & \uD(t,s) \\ 
               \partial_1 \uN(t\rlim,s) & \partial_1 \uD(t\rlim,s) \bigstrut
               \end{pmatrix}
\]
and $\det T_\mu(t,s) = 1$ for all $t\in\R$, 
cf.~\cite[Proposition 2.5]{BenAmorRemling2005} and also \cite[Remark~2.7]{Seifert2011}.
It follows that
\begin{equation}\label{eq:T-inverse}
  T_\mu(s,t) = T_\mu(t,s)^{-1}
  = \begin{pmatrix}
       \partial_1 \uD(t\rlim,s) & -\uD(t,s) \\
      -\partial_1 \uN(t\rlim,s) &  \uN(t,s) \bigstrut
    \end{pmatrix},
\end{equation}
in particular, $\uD(s,t) = -\uD(t,s)$ for all $t\in\R$.
As a shorthand notation we will also write
\[
  \uD(t) := \uD(t,0), \qquad \uN(t) := \uN(t,0).
\]
\end{remark}

Next we provide a variation of constants type formula.

\begin{lemma}\label{lem:repr_u}
  Let $s,t\in\R$. Then  
  \[\begin{pmatrix} v(t) \\ v'(t\rlim)\end{pmatrix} = T_{\mu_1}(t,s)\begin{pmatrix} v(s) \\ v'(s\rlim)\end{pmatrix} + \int_s^t T_{\mu_1}(t,r) \begin{pmatrix} 0 \\ u_2(r)\end{pmatrix}\,d\nu(r).\]
\end{lemma}

\begin{proof}
(i) We will use the following integration by parts type formula:
let $\mu\in\M$, $u$ a solution of $H_\mu u = 0$ and $\tilde{u}\in W_{1,\loc}^1(\R)$.
Then
\[\int_s^t \tilde{u}(r) u(r)\, d\mu(r) = \tilde{u}(t)u'(t\rlim) - \tilde{u}(s)u'(s\rlim) - \int_s^t \tilde{u}'(r) u'(r)\, dr,\]
see \cite[Lemma~1.3.3]{Seifert2012} (note that there $\int_s^t = \int_{[s,t]}$ for $s<t$, so the boundary values slightly change).

(ii) Without loss of generality let $s=0$.
We are going to show that
\[
 \begin{pmatrix} u_2(0) \\ u_2'(0\rlim)\end{pmatrix} - T_{\mu_1}(t,0)^{-1}\begin{pmatrix} u_2(t) \\ u_2'(t\rlim)\end{pmatrix} = \int_0^t T_{\mu_1}(r,0)^{-1}\begin{pmatrix}  0\\ u_2(r)\end{pmatrix}\, d\nu(r).
\]
Multiplying this identity by $T_{\mu_1}(t,0)$ yields the assertion since
\[
  T_{\mu_1}(t,0) \begin{pmatrix} v(0) \\ v'(0\rlim)\end{pmatrix}
  = \begin{pmatrix} u_1(t) \\ u_1'(t\rlim)\end{pmatrix}
    - T_{\mu_1}(t,0) \begin{pmatrix} u_2(0) \\ u_2'(0\rlim)\end{pmatrix}
\]
and $T_{\mu_1}(t,0) T_{\mu_1}(r,0)^{-1} = T_{\mu_1}(t,r)$.

(iii)
Using~\eqref{eq:T-inverse} we compute
\[
  \int_0^t T_{\mu_1}(r,0)^{-1}\begin{pmatrix} 0 \\ u_2(r)\end{pmatrix}\, d\nu(r)
  = \begin{pmatrix}
    -\!\int_0^t \uD1(r) u_2(r)\,d\nu(r) \\
       \int_0^t \uN1(r) u_2(r)\,d\nu(r) \intstrut
    \end{pmatrix}.
\]
Integrating by parts as in (i) we obtain
\begin{align*}
  \begin{pmatrix}
  -\!\int_0^t \uD1 u_2\,d\nu \\
     \int_0^t \uN1 u_2\,d\nu \intstrut
  \end{pmatrix}
& = \begin{pmatrix}
    -\!\int_0^t u_2 \uD1\,d\mu_1 + \int_0^t \uD1 u_2\,d\mu_2 \\
       \int_0^t u_2 \uN1\,d\mu_1 - \int_0^t \uN1 u_2\,d\mu_2 \intstrut
  \end{pmatrix} \\[0.5ex]
& =
  \begin{pmatrix}
   -\bigl( u_2(t) \uD1'(t\rlim) - u_2(0) \bigr) + \bigl( \uD1(t) u_2'(t\rlim) - 0 \bigr) \\
    \bigl( u_2(t) \uN1'(t\rlim) - 0 \bigr) - \bigl( \uN1(t) u_2'(t\rlim) - u_2'(0\rlim) \bigr) \intstrut
  \end{pmatrix} \\[0.5ex]
& =
  \begin{pmatrix}
    u_2(0) \\
    u_2'(0\rlim) \bigstrut
  \end{pmatrix} - \begin{pmatrix}
                    \uD1'(t\rlim) & -\uD1(t) \\
                    -\uN1'(t\rlim) & \uN1(t) \bigstrut
                  \end{pmatrix} \begin{pmatrix} u_2(t) \\ u_2'(t\rlim) \bigstrut\end{pmatrix},
\end{align*}
which by~\eqref{eq:T-inverse} completes the proof.
\end{proof}

\begin{lemma}
\label{lem:u}
Let $c\in\R$. Suppose that $u_2(0) = u_1(0)$, $u_2'(0\rlim) = u_1'(0\rlim) + cu_1(0)$. Then
\[
  v(t) = \int_0^t \! \tfrac{d}{ds} \bigl(\uD1(t,s)u_2(s)\bigr) \cdot (c-\phi_\nu(s))\, ds
  \qquad (t\in\R).
\]
\end{lemma}

\begin{proof}
Let $t\in\R$. Note that $v(0)=0$, $v'(0\rlim) = -\ckern u_2(0)$.
By Lemma~\ref{lem:repr_u} we thus obtain
\[
  v(t) = - \uD1(t) \ckern u_2(0) + \int_0^t \uD1(t,r)u_2(r)\, d\nu(r).
\]
Since $\uD1(t,t) = 0$, we have
\[
  \uD1(t,r)u_2(r) = -\int_r^t \! \tfrac{d}{ds} \bigl(\uD1(t,s)u_2(s)\bigr)\, ds.
\]
By Fubini's theorem it follows from the above two identities that
\begin{align*}
v(t) 
 & = -\uD1(t) \ckern u_2(0) - \int_0^t \! \int_0^s d\nu(r)\, \tfrac{d}{ds} \bigl(\uD1(t,s)u_2(s)\bigr)\, ds \\
 & = \int_0^t (c-\phi_\nu(s)) \, \tfrac{d}{ds} \bigl(\uD1(t,s)u_2(s)\bigr)\, ds.
 \qedhere
\end{align*}
\end{proof}

Now we can prove the desired estimate on $v = u_1-u_2$ in terms of weak seminorms of $\nu = \mu_1-\mu_2$.

\begin{proposition}\label{prop:stability}
Suppose that $u_2(0) = u_1(0)$, $u_2'(0\rlim) = u_1'(0\rlim) + c_{\nu,0} 
u_1(0)$, where $c_{\nu,0}$ is as in~\eqref{eq:normalternative}.
Let $\alpha,\beta\in\Z$, $\alpha\le-1$, $\beta\ge1$.
Let $c,\omega>0$ such that
\begin{equation}\label{eq:exp-bound}
  \abs{\partial_1 \uD1(t\rlim,s)} \le \ckern e^{\omega\abs{t-s}} \qquad (s,t\in\R).
\end{equation}
Then
\[
  \abs{v(t)} \le C\ckern e^{\omega\abs{t}} \norm{u_2|_{[\alpha,\beta]}}_\infty \norm{\nu}_{[\alpha,\beta]} \qquad(t\in[\alpha,\beta]),
\]
with a constant $C>0$ depending only on $\omega$ and $\norm{\mu_2}_\lu$.
\end{proposition}

Observe that~\eqref{eq:exp-bound} is always satisfied with
$\omega = \norm{\mu_1}_\lu + 1$ and $c=e^\omega$; this follows from Lemma~\ref{lem:expo}
since $\uD1(s,s) = 0$ and $\partial_1 \uD1(s\rlim,s) = 1$ for all $s\in\R$.

\begin{proof}
From $\uD1(s,s) = 0$ and the assumed estimate on $\partial_1\uD1$ we derive
$\abs{\uD1(t,s)} \le \rfrac{c}{\omega} e^{\omega\abs{t-s}}$ for all $s,t\in\R$.
Moreover, by Lemma~\ref{lem:u'_leq_u} we have
$\norm{u_2'|_{[\alpha,\beta]}}_\infty \le (\norm{\mu_2}_\lu+2) \norm{u_2|_{[\alpha,\beta]}}_\infty$.
Recalling from~\eqref{eq:T-inverse} that $\uD1(t,s) = -\uD1(s,t)$, we thus obtain
\begin{equation}\label{eq:der-est}
\begin{split}
\abs{\smash{\tfrac{d^+}{ds}} \bigl(\uD1(t,s)u_2(s)\bigr)}
 &= \abs*{ -\partial_1 \uD1(s\rlim,t) u_2(s) + \uD1(t,s)u_2'(s\rlim) } \\
 &\le C_0 \ckern e^{\omega\abs{t-s}} \norm{u_2|_{[\alpha,\beta]}}_\infty
\end{split}
\end{equation}
for all $s,t\in[\alpha,\beta]$, with $C_0 = 1 + \rfrac{1}{\omega} (\norm{\mu_2}_\lu+2)$.

Let $t\in[0,\beta]$. By Lemma~\ref{lem:u} we have
\[
  \abs{v(t)} \le \int_0^t \abs{\smash{\tfrac{d^+}{ds}}
\bigl(\uD1(t,s)u_2(s)\bigr)} \cdot \abs{\phi_\nu(s)-c_{\nu,0}}\, ds.
\]
By~\eqref{eq:der-est} and Lemma~\ref{lem:est_1} we conclude that
\begin{align*} 
\abs{v(t)}
& \le C_0\ckern \norm{u_2|_{[\alpha,\beta]}}_\infty \sum_{k=1}^{\beta} \int_{k-1}^{k} 
e^{\omega(t-s)} \abs{\phi_\nu(s)-c_{\nu,0}}\, ds \\
& \le C_0\ckern \norm{u_2|_{[\alpha,\beta]}}_\infty \sum_{k=1}^{\beta} e^{\omega(t+1-k)} \cdot 2k \norm{\nu}_{[\alpha,\beta]}
  \le C\ckern \norm{u_2|_{[\alpha,\beta]}}_\infty e^{\omega t} \norm{\nu}_{[\alpha,\beta]},
\end{align*}
with $C := C_0 \sum_{k=1}^{\infty} 2k e^{-\omega(k-1)}$.
The proof in the case $t\in[\alpha,0)$ is analogous.
\end{proof}

Using the previous result we can sharpen the estimate in Lemma~\ref{lem:expo}.

\begin{proposition}
\label{prop:exp_growth}
	Let $\mu\in\M$ and let $u$ be a solution of the equation $H_\mu u=0$.
	Then for $t\in\R$ one has
	\[
	  \bigl(\norm{\mu}_\lu\abs{u(t)}^2 + \abs{u'(t\rlim)}^2\bigr)^{1/2}
	  \le \bigl( \norm{\mu}_\lu\abs{u(0)}^2 + \abs{u'(0\rlim)}^2 \bigr)^{1/2}
	      e^{\norm{\mu}_\lu^{1/2}(|t|+1/2)}.
	\]
\end{proposition}

\begin{proof}
    The case $\mu=0$ is trivial, since then $u$ is affine. So, let 
    $\mu\neq 0$.

    (i)
    First, assume that $\mu = \rho\lambda$ with a density $\rho\in C(\R)$.
    Then $u\in C^2(\R)$, $u'' = \rho u$. Let $\omega>0$. 
    Then for $\phi(t):= \omega^2\abs{u(t)}^2 + \abs{u'(t)}^2$ we obtain
    \begin{align*}
    \abs{\phi'(t)}
      &=   \abs{ 2\Re\bigl((\omega^2+\rho)u(t) \overline{u'(t)\mkern-2mu}\mkern2mu \bigr) } \\
      &\le (\omega^2+\abs{\rho}) \bigl( \omega\abs{u(t)}^2 + \rfrac1\omega \abs{u'(t)}^2 \bigr)
       =   (\omega+\rfrac1\omega\abs{\rho}) \phi(t).
    \end{align*}
    It follows that $\phi(t) \le \phi(s) \exp\bigl( \omega(t-s) + \rfrac1\omega \int_s^t \abs{\rho(r)}\,dr \bigr)$
    and hence
    \[
      \omega^2\abs{u(t)}^2 + \abs{u'(t)}^2
      \le \bigl(\omega^2\abs{u(s)}^2 + \abs{u'(s)}^2\bigr)
          \exp\bigl(\omega(t-s) + \rfrac{1}{\omega}\abs{\mu}([s,t])\bigr)
    \]
    for all $s,t\in\R$, $s<t$.
    
    (ii)
    By Proposition~\ref{prop:approx} there exists a sequence $(\mu_n)_{n\in\N}$ in $\M$
    such that $\mu_n$ has a smooth density and $\norm{\mu_n}_\lu\le \norm{\mu}_\lu$
    for all $n\in\N$, $\norm{\mu_n-\mu}_\R\to 0$ and
    $\limsup_{n\to\infty}\abs{\mu_n}(I)\le \abs{\mu}(I)$ for all compact intervals $I\subseteq \R$.
    Lemma~\ref{lem:vague} implies $\mu_n\to\mu$ vaguely and, hence, $\1_{[a,b]}\mu_n \to \1_{[a,b]}\mu$ weakly for all $a,b\in\R$ such that $\mu(\set{a}) = \mu(\set{b}) = 0$.
    (Note that $\mu$ has at most countable many atoms.)

    (iii)
    For $n\in\N$ let $u_n$ be the solution of $H_{\mu_n} u_n = 0$ such that
    $u_n(0) = u(0)$ and $u_n'(0\rlim) = u'(0\rlim) + c_{\mu-\mu_n,0}u(0)$.
    Then $(u_n)$ is uniformly bounded on any compact interval
    by Lemma~\ref{lem:expo} and Remark~\ref{rem:c_0}(a), and
    Proposition~\ref{prop:stability} implies that $u_n\to u$ locally uniformly.
    Let $s,t\in\R$ such that $\mu(\set{s}) = \mu(\set{t}) = 0$. Then
    \[
      u_n'(t) - u_n'(s)
        = \int_s^t u_n(r)\,d\mu_n(r)
      \to \int_s^t u(r)\,d\mu(r)
        = u'(t\rlim) - u'(s\rlim)
    \]
    as $n\to\infty$ since $u_n\to u$ uniformly on $[s,t]$ and $\1_{[s,t]}\mu_n\to \1_{[s,t]}\mu$ weakly.
    It follows from Lemma~\ref{lem:u'_leq_u} that $(u_n')$ is uniformly bounded on $[0,1]$.
    Thus, integrating $s$ from $0$ to $1$ we conclude by
    Lebesgue's dominated convergence theorem that
    \[u_n'(t) - \bigl(u_n(1) - u_n(0)\bigr)
      \to u'(t\rlim) - \bigl(u(1) - u(0)\bigr)\]
    and hence $u_n'(t) \to u'(t\rlim)$.
    
    (iv)
    Let $t>s>0$ such that $\mu(\set{s}) = \mu(\set{t}) = 0$. By~(i) we obtain
    \[
      \omega^2\abs{u_n(t)}^2 + \abs{u_n'(t)}^2
      \le \bigl(\omega^2\abs{u_n(s)}^2 + \abs{u_n'(s)}^2\bigr)
          \exp\bigl(\omega(t-s) + \rfrac{1}{\omega}\abs{\mu_n}([s,t])\bigr)
    \]
    for all $n\in\N$. For $n\to\infty$ we infer, using part~(iii) and the estimate
    $\limsup\abs{\mu_n}([s,t]) \le \abs{\mu}([s,t])$, that
    \[
      \omega^2\abs{u(t)}^2 + \abs{u'(t\rlim)}^2
      \le \bigl(\omega^2\abs{u(s)}^2 + \abs{u'(s\rlim)}^2\bigr)
          \exp\bigl(\omega(t-s) + \rfrac{1}{\omega}\abs{\mu}([s,t])\bigr).
    \]
    Finally, let $t>0$.
    Then there exist sequences $(s_n)$ in $[0,t)$ and $(t_n)$ in $[t,\infty)$
    such that $s_n\to 0$, $t_n\to t$ and $\mu(\set{s_n}) = \mu(\set{t_n}) = 0$
    for all $n\in\N$, and from the above we deduce that
    \[
      \omega^2\abs{u(t)}^2 + \abs{u'(t\rlim)}^2
      \le \bigl(\omega^2\abs{u(0)}^2 + \abs{u'(0\rlim)}^2\bigr)
          \exp\bigl(\omega t + \rfrac{1}{\omega}\abs{\mu}((0,t])\bigr).
    \]
    Choosing $\omega=\norm{\mu}_\lu^{1/2}$ yields the assertion for $t\ge0$
    since then $\abs{\mu}((0,t]) \le \omega^2(t+1)$. The case $t<0$ is proved analogously.
\end{proof}

\section{Gordon's Theorem}

We are now in the position to formulate the weak Gordon condition and to prove our main theorem. We will then apply the result to quasiperiodic measures.

\begin{definition}
   Let $\mu\in\M$ and 
   $C> 0$. We say that $\mu$ is a \emph{weak Gordon measure} with weight $C$
  if there exists a sequence $(p_m)$ in $(0,\infty)$ with $p_m\to \infty$ such 
that
  \[\lim_{m\to\infty} e^{Cp_m} \norm{\mu - \mu(\cdot+p_m)}_{[-p_m,p_m]} = 0.\]
\end{definition}

The following lemma relates our notion of weak Gordon measures with previous versions for Gordon potentials/measures (see \cite{Gordon1976,DamanikStolz2000,Seifert2011}).

\begin{lemma}
\label{lem:equiv_Gordon}
  Let $\mu\in\M$, $C> 0$. Then the following are equivalent:
  \begin{enumerate}[label=\textup{(\roman*)}]
    \item
      $\mu$ is a weak Gordon measure with weight $C$.
    \item
      There exist sequences $(p_m)$ in $(0,\infty)$ and $(\mu_m)$ in $\M$ such that
      $\mu_m$ is periodic with period $p_m$ ($m\in\N$), $p_m\to\infty$ and
      \[
        e^{Cp_m} \norm{\mu-\mu_m}_{[-p_m,2p_m]} \to 0 \qquad (m\to\infty).
      \]
  \end{enumerate}
  Moreover, the measures $\mu_m$ in~\textup{(ii)} can be chosen such that
  \[
    \1_{[a_m,p_m-a_m]} \mu_m = \1_{[a_m,p_m-a_m]} \mu, \quad
    \norm{\mu_m}_\lu \le \bigl(1+\tfrac{1}{2a_m}\bigr)\norm{\mu}_\lu
  \]
  for all $m\in\N$, with $0 < a_m \le \frac{p_m}{2}$ and $\inf_{m\in\N} a_m>0$.
\end{lemma}

\begin{proof}
  (i) $\Rightarrow$ (ii): 
  Choose a sequence $(p_m)$ as in the above definition,
  and let $(a_m)$ be as in the additional statement of the lemma.
  Let $m\in\N$, and let $\psi_m\from\R\to\R$ be the piecewise affine function with
  $\spt\psi_m = [-a_m,p_m+a_m]$, $\psi_m = 1$ on $[a_m,p_m-a_m]$ and 
  $\|\psi'\|_\infty = \frac{1}{2a_m}$.
  Then $\sum_{k\in\Z} \psi_m(\cdot+kp_m) = \1_\R$, the measure
  \[
    \mu_m := \sum_{k\in\Z} (\psi_m\mu)(\cdot+kp_m)
  \]
  is $p_m$-periodic and $\1_{[a_m,p_m-a_m]} \mu_m = \1_{[a_m,p_m-a_m]} \mu$.
  
  With $I_m := [-p_m,2p_m]$ we obtain
  \begin{align*}
  \norm{\mu-\mu_m}_{I_m}
   &\le \sum_{k\in\Z} \norm{\psi_m(\cdot+kp_m)\bigl(\mu-\mu(\cdot+kp_m)\bigr)}_{I_m} \\
   &\le \bigl(1+\tfrac{1}{2a_m}\bigr)
        \sum_{k\in\Z} \norm{\mu-\mu(\cdot+kp_m)}_{I_m\cap[-a_m-kp_m,p_m+a_m-kp_m]}
  \end{align*}
  by Lemma~\ref{lem:norm_psi_mu}.
  For $|k|>2$ the set $I_m\cap[-a_m-kp_m,p_m+a_m-kp_m]$ is empty;
  for $|k|\le2$ one easily shows the estimate
  \[
    \norm{\mu-\mu(\cdot+kp_m)}_{I_m\cap[-a_m-kp_m,p_m+a_m-kp_m]}
    \le |k| \cdot \norm{\mu-\mu(\cdot+p_m)}_{[-p_m,p_m]},
  \]
  where for $|k|=2$ one has to use the triangle inequality.
  It follows that $\norm{\mu-\mu_m}_{I_m} \le
  6 \bigl(1+\tfrac{1}{2a_{m\vphantom{p}}}\bigr) \norm{\mu-\mu(\cdot+p_m)}_{[-p_m,p_m]}$,
  which implies~(ii).

  Finally, assume without loss of generality that $p_m\ge2$ for all $m\in\N$.
  Then for all $a\in\R$ we have
  \[
    \abs{\mu_m}\bigl((a,a+1]\bigr)
    \le \sum_{k\in\Z} \norm{\psi_m(\cdot+kp_m)|_{(a,a+1]}}_\infty \norm{\mu}_\lu
    \le \bigl(1+\tfrac{1}{2a_m}\bigr)\norm{\mu}_\lu,
  \]
  due to the particular construction of $\psi_m$ and the fact that at most 3 summands in the second term are nonzero.

  (ii) $\Rightarrow$ (i): Since $\mu_m$ is $p_m$-periodic, we have
  \begin{align*}
  \MoveEqLeft \norm{\mu- \mu(\cdot+p_m)}_{[-p_m,p_m]} \\
   &\le \norm{\mu - \mu_m}_{[-p_m,p_m]}
      + \norm{\mu_m(\cdot+p_m) - \mu(\cdot+p_m)}_{[-p_m,p_m]} \\
   &\le 2\norm{\mu - \mu_m}_{[-p_m,2p_m]},
  \end{align*}
  and the assertion follows.
\end{proof}

\begin{remark}
The above lemma shows that a weak Gordon measure can (locally on three periods)
be very well approximated by periodic measures (where the approximation depends on $C$).
In \cite{Seifert2011} a local measure was called Gordon measure if the condition of Lemma~\ref{lem:equiv_Gordon}(ii) is satisfied for all $C>0$, 
where $\norm{\cdot}_{[-p_m,2p_{m\vphantom{p}}]}$ is replaced by the total variation norm on $[-p_m,2p_m]$.

The definition in \cite{Seifert2011} extends the previous definitions for
bounded (\cite{Gordon1976}) and locally integrable potentials (\cite{DamanikStolz2000}),
so our notion generalizes and refines the Gordon condition from those papers.
\end{remark}

We will need the following estimate for solutions to periodic measures.

\begin{lemma}[{see \cite[Lemma~2.8]{Seifert2011}}]
\label{lem:periodic}
	Let $\mu$ be $p$-periodic and $z\in \C$. Let $u$ be a solution of $H_\mu u = z u$. Then
	\[
	  \norm{\begin{pmatrix} u(0) \\u'(0\rlim)\end{pmatrix}}
	  \le 2 \max\set{\norm{\begin{pmatrix} u(t) \\u'(t\rlim)\end{pmatrix}} ;\; t\in\set{-p,p,2p} }.
        \]
\end{lemma}

The following is our main result on absence of eigenvalues.

\begin{theorem}\label{thm:Gordon}
  Let $\mu\in\M$ be a weak Gordon measure with weight\/ $C>0$. Then $H_\mu$
  does not have any eigenvalues with modulus less than $C^2-\norm{\mu}_\lu$.
\end{theorem}

\begin{proof}
  Let $(p_m)$ and $(\mu_m)$ be as in Lemma~\ref{lem:equiv_Gordon}(ii),
  without loss of generality $p_m\ge 2$ for all $m\in\N$.
  Let the additional statement of the lemma be satisfied with a sequence $(a_m)$ satisfying
  $p_m+a_m\in\N$ for all $m\in\N$, $a_m\to\infty$ and $\rfrac{a_m}{p_m}\to 0$.

  Assume that $z\in \C$ with $\abs{z} < C^2-\norm{\mu}_\lu$ is an eigenvalue of $H_\mu$.
  Let $u$ be a corresponding eigenfunction; then $u\in W_2^1(\R)\subseteq C_0(\R)$.

  For $m\in\N$ let $u_m$ be the solution of $H_{\mu_m} u_m = zu_m$ satisfying 
  $u_m(a_m) = u(a_m)$, $u_m'(a_m\rlim) =  u'(a_m\rlim)$.
  Then $u_m=u$ on $[a_m,p_m-a_m]$ since $\mu_m=\mu$ on this interval.
  We now apply Proposition~\ref{prop:stability} with $\nu = (\mu_m-z\lambda) - (\mu-z\lambda)$.
  Note that $\1_{[a_m,a_m+2]}\nu = 0$, so
  one has $c_{\nu,a_{m}+1} = 0$ in~\eqref{eq:normalternative}.
  Setting $\omega_m := \norm{\mu_m-z\lambda}_\lu^{1/2}$ and taking into account
  Proposition~\ref{prop:exp_growth}, we obtain
  \[
    \abs{u(t)-u_m(t)} \le C_m e^{\omega_m\abs{t-(a_m+1)}}\norm{\mu-\mu_m}_{[-p_m,a_m+1]}
    \quad \bigl( m\in\N,\ t\in[-p_m,a_m] \bigr)
  \]
  with $C_m>0$ depending only on $\omega_m$, and similarly for $t\in[p_m-a_m,2p_m]$.
  Therefore,
  \begin{equation}\label{eq:uum-est}
    \sup_{t\in[-p_m,2p_m]}\abs{(u-u_m)(t)}
    \le C_m e^{\omega_m(p_m+a_m+1)}\norm{\mu-\mu_m}_{[-p_m,2p_m]}.
  \end{equation}
  
  By Lemma~\ref{lem:equiv_Gordon} and the assumption on $z$ we have
  \[
    \omega_m^2 
    \le \norm{\mu_m}_\lu + \abs{z}
    \le \bigl(1+\tfrac{1}{2a_{m\vphantom{p}}}\bigr) \norm{\mu}_\lu + \abs{z}
    \longrightarrow \norm{\mu}_\lu + \abs{z}
    < C^2.
  \]
  Hence, for large $m$ we obtain
  \[
    \omega_m(p_m+a_m+1) = \omega_m p_m \bigl(1+\rfrac{a_m+1}{p_m}\bigr) \leq Cp_m.
  \]
  We conclude that the right hand side in~\eqref{eq:uum-est} tends to zero as $m\to\infty$.
  Therefore, given $\eps>0$
  there exists $m_0\in\N$ such that $\abs{u(t)-u_m(t)}\le \eps$ for all
  $m\ge m_0$ and $t\in[-p_m,2p_m]$. Since $u \in C_0(\R)$, we can choose
  $m_0$ such that $\abs{u(t)} \le \eps$ for $\abs{t} \ge p_{m_0}-1 =: t_0$.
  Then $\abs{u_m} \le 2\eps$ on $[-p_m,2p_m] \setminus (-t_0,t_0)$.
  By Lemma~\ref{lem:u'_leq_u} it follows that $\abs{u_m'} \le 2\eps(2+\norm{\mu_m}_\lu)$ on that set.
  
  Thus we have shown
  \[
    \bigl(u_m(\pm p_m),u_m'(\pm p_m\rlim)\bigr),
    \bigl(u_m(2p_m),u_m'(2p_m\rlim)\bigr) \to 0 \qquad (m\to\infty).
  \]
  By Lemma~\ref{lem:periodic} we conclude that
  $\bigl(u_m(0),u_m'(0\rlim)\bigr) \to 0$ as $m\to\infty$.
  For any compact interval $I\subseteq\R$,
  Lemma~\ref{lem:expo} yields $u_m\to 0$ uniformly on $I$.
  Since $u_m\to u$ uniformly on $I$ by~\eqref{eq:uum-est},
  we obtain $u = 0$, a contradiction.
\end{proof}

As an immediate consequence of Theorem~\ref{thm:Gordon} we obtain the following version of Gordon's theorem.

\begin{corollary}
\label{cor:no_ev}
  Let $\mu$ be a weak Gordon measure with weight\/ $C$ for all\/ $C>0$.
  Then $H_\mu$ does not have any eigenvalues.
\end{corollary}

By a scaling argument we can improve the eigenvalue bound $C^2-\norm{\mu}_\lu$
from Theorem~\ref{thm:Gordon}.

\begin{corollary}
\label{cor:scaling}
  Let $\mu\in\M$ be a weak Gordon measure with weight\/ $C>0$, and let $r>0$.
  Then $H_\mu$ does not have any eigenvalues with modulus less than
  $C^2-\norm{\mu}_{\lu,r}$, where 
  \[
    \norm{\mu}_{\lu,r} := \smash{\rfrac1r \sup_{a\in\R} \abs{\mu}((a,a+r])}.
  \]
  (Note that $\norm{\mu}_{\lu,1} = \norm{\mu}_\lu$.)
\end{corollary}

\begin{proof}
Let $\mu_r := r\mu(r\mkern1mu\cdot)$,
and define $B_r \from L_2(\R) \to L_2(\R)$ by $B_ru := u(\frac\cdot r)$.
A~straightforward computation shows that $B_r^{-1}H_\mu B_r = r^{-2} H_{\mu_r}$.

We now show that $\mu_r$ is a weak Gordon measure with weight $Cr$.
Let $(p_m)$ be a sequence of periods for the weak Gordon measure $\mu$.
For $u\in W_{\!\infty}^1(\R)$ with $\spt u \subseteq [-\rfrac{p_m}{r},\rfrac{p_m}{r}]$,
$\diam\spt u \le 2$ and $\norm{u'}_\infty\le 1$ we have
\[
  \int_{-p_m/r}^{p_m/r} u\, d\bigl(\mu_r-\mu_r(\cdot+\rfrac{p_m}{r})\bigr)
  = \int_{-p_m}^{p_m} ru(\rfrac\cdot r)\, d\bigl(\mu-\mu(\cdot+p_m)\bigr).
\]
Since $u_r := ru(\rfrac\cdot r)$ satisfies $\diam\spt u_r \le 2r$ and
$\norm{u_r'}_\infty\le 1$, we infer by Remark~\ref{rem:norm_spt} that
\[
  \norm{\mu_r-\mu_r(\cdot+\rfrac{p_m}{r})}_{[-p_m/r,p_m/r]}
  \le (r+1)^2 \norm{\mu-\mu(\cdot+p_m)}_{[-p_m,p_m]}.
\]
Multiplying by $\exp(Cr\rfrac{p_m}{r})$ we deduce that
$\mu_r$ is indeed a weak Gordon measure with weight $Cr$.
Therefore, by Theorem~\ref{thm:Gordon}, $H_{\mu_r}$ has no eigenvalues $z$ with
$|z| < (Cr)^2 - \norm{\mu_r}_\lu$.

To conclude the assertion for $H_\mu = r^{-2} B_r H_{\mu_r} B_r^{-1}$,
it remains to observe that $\norm{\mu_r}_\lu = r^2 \norm{\mu}_{\lu,r}$.
\end{proof}

\begin{remark}
\label{rem:improvements}
(a) The supremum of all $C>0$ such that $\mu$ is a weak Gordon measure with weight $C$ is given by
\[
  C_\mu := -\liminf_{p\to\infty} \rfrac1p \ln \norm{\mu - \mu(\cdot+p)}_{[-p,p]}
  \vspace{-1ex}
\]
if $C_\mu>0$.

(b) Assume that $C_\mu>0$. Then by part~(a) and Corollary~\ref{cor:scaling},
$H_\mu$ does not have any eigenvalues $z$ with
\[
  |z| < E_\mu := C_\mu^2 - \inf_{r>0} \norm{\mu}_{\lu,r}.
\]
It is easy to see that
$\norm{\mu}_{\lu,r} \le (1+\rfrac sr) \norm{\mu}_{\lu,s}$ for all $r>s>0$.
Thus, $\lim_{r\to\infty} \norm{\mu}_{\lu,r} = \inf_{r>0} \norm{\mu}_{\lu,r}$.

In the next section we will show that $-E_\mu$ can occur as an eigenvalue,
so the above bound is sharp.

(c) In the proof of Theorem~\ref{thm:Gordon} we have shown that $H_\mu u = zu$
does not have any solutions in $C_0(\R)$, for $z$ with small modulus.
Hence, we also obtain absence of eigenvalues for $H_\mu$
considered as an operator in $L_p(\R)$ with $1\le p < \infty$.
\end{remark}

We now apply our main result to quasiperiodic measures.

\begin{example}
  Let $\mu_1,\mu_2\in\M$ be $1$- and $\alpha$-periodic, respectively, where $\alpha>0$ is irrational.
  Then $\mu := \mu_1+\mu_2$ is a quasiperiodic measure.
  Under a suitable assumption on $\alpha$ we are going to prove that $H_\mu$ has no eigenvalues.
  Analogous results have been shown in \cite{DamanikStolz2000} under an additional assumption
  on the oscillation of $V_2$, where $\mu_2=V_2\lambda$, and in \cite{Seifert2011}
  under a H\"older condition on the translates of $\mu_2$.
  Using our weak Gordon condition, we can remove these additional assumptions on $\mu_2$ altogether.

  As in \cite{DamanikStolz2000,Seifert2011} we assume that
  there exists $B>0$ and a sequence $(\rfrac{p_m}{q_m})$ in $\Q$ such that
  \[
    \abs{\alpha - \frac{p_m}{q_m}} \le Bm^{-q_m} \qquad (m\in\N);
  \]
  in particular, $\alpha$ is a Liouville number.
  As for the Liouville numbers one sees that the set of all numbers $\alpha$
  with the above property is a dense $G_\delta$ set.

  Using the periodicity of $\mu_1,\mu_2$ and Remark~\ref{rem:norm-comp}(b),
  we now show that $\mu$ is a weak Gordon measure with weight $C$, for all $C>0$:
  \begin{align*}
  \norm{\mu-\mu(\cdot+p_m)}_{[-p_m,p_m]}
  &= \norm{\mu_2-\mu_2(\cdot+p_m-\alpha q_m)}_{[-p_m,p_m]} \\
  &\le 3\abs{p_m-\alpha q_m} \norm{\mu_2}_\lu
    = 3q_m\abs{\rfrac{p_m}{q_m}-\alpha} \norm{\mu_2}_\lu
  \end{align*}
  and hence
  \[
    e^{Cp_m} \norm{\mu-\mu(\cdot+p_m)}_{[-p_m,p_m]}
    \le \exp\bigl(Cq_m\rfrac{p_m}{q_m}\bigr) \cdot 3q_m Bm^{-q_m} \norm{\mu_2}_\lu \to 0
  \]
  as $m\to\infty$. By Corollary~\ref{cor:no_ev} it follows that $H_\mu$ does not have any eigenvalues.
\end{example}

\section{Sharpness of the condition}

In this section we show that our results are sharp in the following sense:
we construct a measure $\mu\in\M$ such that the operator $H_\mu$ has $-1$ as
an eigenvalue, $\mu$~is a weak Gordon measure with weight $C$ for all $C<1$,
and $\inf_{r>0} \norm{\mu}_{\lu,r} = 0$.
With the notation of Remark~\ref{rem:improvements} we then have
$C_\mu = 1$ and $E_\mu = 1$.

The construction will be such that the Gordon condition is satisfied
with arbitrarily many periods, not just with three periods.
In fact, we construct a symmetric eigenfunction to the eigenvalue $-1$ in the following way.
We first fix $u$ on an infinite discrete set $T\subseteq\R$ and then choose
the continuous continuation to $\R$ satisfying $u'' = u$ on $\R\setminus T$.
We will determine $T$ and also $u$ on~$T$ recursively.

Choose a sequence $(l_m)_{m\in\N}$ in $(0,\infty)$ satisfying $l_m\to \infty$.
Set $p_0 := 0$, $T_0 := \set{0}$ and $u(0) := 1$. Recursively we define, for $m\in\N$,
\begin{align*}
p_m & := 2(m-1)p_{m-1}+l_m \\ 
T_m &:= \set{jp_m+t ;\; |j|\le m,\ t\in T_{m-1}}  \cap[-mp_m,mp_m], \\[1ex]
u(jp_m+t) &:= 2^{-m|j|} u(t) \qquad \bigl(jp_m+t\in T_m\setminus T_{m-1}\bigr).
\end{align*}
Let $T:=\bigcup_{m\in\N_0} T_m$. Then $u$ is fixed on the discrete set $T$.
Consider the unique continuous continuation $u\from\R\to(0,\infty)$ satisfying $u''=u$ on $\R\setminus T$.
It is easy to see that $u\in L_2(\R)$.
Moreover, note that the derivative of $u$ has jumps precisely at $T$.
Thus, $u$ solves $-u''+\mu u = -u$ for some pure-point measure $\mu$ with $\spt\mu = T$.

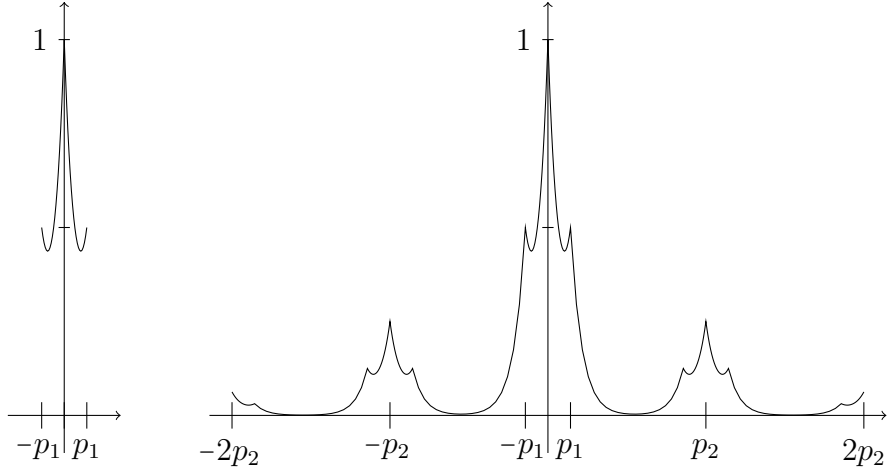
\begin{figure}[htb]
\renewcommand\-{%
    \settowidth{\dimen0}{$0$}%
    \!\smash{\resizebox{1.1\dimen0}{\height}{$-$}}\mkern0.5mu%
}
\centering
\begin{tikzpicture}[xscale=0.15,yscale=5]
\draw[->] (-5,0)--(5,0);
\draw[->] (0,-0.1)--(0,1.1);

\foreach \x/\label in {-2/$\-p_1$,2/$p_1$,0/\vphantom{$2p_2$}}
     		\draw (\x,1pt) -- (\x,-1pt)
			node[anchor=north] {\label};

\draw (0.5,1) -- (-0.5,1) node[left]{$1$};
\draw (0.5,0.5) -- (-0.5,0.5);
			
\draw plot file {plot_data1.data};
\end{tikzpicture}
\qquad
\begin{tikzpicture}[xscale=0.15,yscale=5]
\draw[->] (-30,0)--(30,0);
\draw[->] (0,-0.1)--(0,1.1);

\foreach \x/\label in {-28/$\-2p_2$,-14/$\-p_2$,-2/$\-p_1$,2/$p_1$,14/$p_2$,28/$2p_2$}
     		\draw (\x,1pt) -- (\x,-1pt)
			node[anchor=north] {\label};
			
\draw (0.5,1) -- (-0.5,1) node[left]{$1$};
\draw (0.5,0.5) -- (-0.5,0.5);

\draw plot file {plot_data2.data};
\end{tikzpicture}
\caption{$u$ defined on the intervals corresponding to $T_1$ and to $T_2$.}
\end{figure}

We now assume that
\begin{equation}\label{eq:lk-cond}
  \frac{2(m-1)p_{m-1}+m\ln 2}{l_m} \to 0 \qquad (m\to\infty).
\end{equation}
Since then $2^m \le e^{l_m}$ for large $m$, a straightforward computation yields boundedness of the masses of single points (see also Remark \ref{rem:mu_lokbeschr} below).
Furthermore, $T=\spt\mu$ is uniformly discrete (i.e., there exists a minimal distance between any two points in $T$). These two facts imply $\mu\in\M$.
Hence, $-1$ is an eigenvalue of $H_\mu$ (with eigenfunction $u$).

We will now investigate $\mu$ and show that $\mu$ is indeed a weak Gordon measure with weight $C$ for all $C<1$.
Let $s_m := (m-1)p_{m-1}$ and $t_m := s_m+l_m = p_m-s_m$.
By construction, the measures 
\[
  \1_{(-mp_m,(m-1)p_m)}\mu \quad\text{and}\quad
  \1_{(-mp_m,(m-1)p_m)}\bigl(\mu(\cdot+p_m)\bigr)
\]
differ only at the points $-s_m$ and $-t_m$, and there the corresponding point masses are interchanged.
(It suffices to consider the open intervals here since the test functions for computing
the $\norm{\cdot}_I$-norm are zero at the boundary.)  
With the help of the next lemma we can compute the difference $\mu(\set{-s_m}) - \mu(\set{-t_m})$.

\begin{lemma}
\label{lem:mu_diff}
  Let $l,L>0$, and let $a,b,c,d>0$ satisfy $\frac{a}{b} = \frac{d}{c}$.
  Let $u\in C[-l,L+l]$ such that $u(-l) = a$, $u(0) = b$, $u(L) = c$ and $u(L+l) = d$.
  Assume that $u$ is twice continuously differentiable on $(-l,L+l)\setminus\set{0,L}$, with $u'' = u$.
  Then the measure $\mu$ defined by $-u''+\mu u = -u$ satisfies $\spt \mu \subseteq\set{0,L}$ and
  \[
    \mu(\set{0}) - \mu(\set{L}) = \frac{2\frac{c}{b}-2\frac{b}{c}}{e^L-e^{-L}}\,.
  \]
\end{lemma}

\begin{proof}
  First observe that $\frac1b u(-t) = \frac1c u(L+t)$ holds for $t=0$ and for $t=l$,
  and hence for all $t\in[-l,0]$ by uniqueness of the boundary value problem for $u''=u$.
  It follows that $\frac{u'(0\llim)}{u(0)} + \frac{u'(L\rlim)}{u(L)} = 0$.
  Therefore,
  \begin{equation}\label{eq:mu-diff}
  \begin{split}
  \mu(\set{0}) - \mu(\set{L})
    &= \frac{u'(0\rlim)-u'(0\llim)}{u(0)} - \frac{u'(L\rlim)-u'(L\llim)}{u(L)} \\
    &= \frac1b\mkern1mu u'(0\rlim) + \frac1c\mkern1mu u'(L\llim).
  \end{split}
  \end{equation}
  A straightforward computation yields
  \begin{equation}\label{eq:usol}
    u(t) = \frac{c-e^{-L}b}{e^{L}-e^{-L}} e^{t} + \frac{e^{L}b-c}{e^{L}-e^{-L}} e^{-t}
  \end{equation}
  for all $t\in[0,L]$, so from~\eqref{eq:mu-diff} we deduce that
  \begin{align*}
  \mu(\set{0}) - \mu(\set{L})
   &= \frac1b \cdot \frac{2c-(e^{L}+e^{-L})b}{e^{L}-e^{-L}} + \frac1c \cdot \frac{(e^{L}+e^{-L})c-2b}{e^{L}-e^{-L}} \\
   &= \frac{2\frac{c}{b}-2\frac{b}{c}}{e^L-e^{-L}}\,.
   \qedhere
  \end{align*}
\end{proof}

\begin{remark}
\label{rem:mu_lokbeschr}
  If $e^{-L} \le \frac{c}{b} \le e^L$, then the two coefficients in the right hand side of~\eqref{eq:usol}
  are non-negative, and this implies $\bigl|\frac{u'(0\rlim)}{u(0)}\bigr|,
  \bigl|\frac{u'(L\llim)}{u(L)}\bigr| \le 1$.
  (In the following we will have $\frac{c}{b} =2^{-m}$.)
\end{remark}

We now apply Lemma \ref{lem:mu_diff} with $l=l_1$ and $L=l_m$, for $m\ge2$.
Note that $u(t_m) = u(1p_m-s_m) = 2^{-m\cdot 1}u(-s_m) = 2^{-m}u(s_m)$. Therefore,
\[
  \mu(\set{s_m}) - \mu(\set{t_m}) = \frac{2\cdot 2^{-m}-2\cdot 2^m}{e^{l_m}-e^{-l_m}}\,,
\]
and for $m\ge2$ with $l_m\ge2$ it follows that
\[
  \norm{\mu-\mu(\cdot+p_m)}_{[-mp_m,(m-1)p_m]}
  = \frac{2\cdot 2^m-2\cdot 2^{-m}}{e^{l_m}-e^{-l_m}}
  \le 4 \mkern1mu e^{m\ln2-l_m}.
\]

For $C<1$ we infer by~\eqref{eq:lk-cond} that
\[
  e^{Cp_m}\norm{\mu-\mu(\cdot+p_m)}_{[-mp_m,(m-1)p_m]} \le 4 e^{C2(m-1)p_{m-1} + m\ln 2 - (1-C)l_m}\to 0.
\]
In particular,
$\mu$ is a weak Gordon measure with weight $C$, and hence $C_\mu\geq 1$.

Furthermore we have $\inf_{r>0} \norm{\mu}_{\lu,r} = 0$
since $\mu\in \M$ and 
\[
  \rfrac1r \sup_{a\in\R} \#\bigl(\spt\mu \cap (a,a+r]\bigr) \to 0
\]
as $r\to\infty$, where $\#$ denotes the counting measure. Thus, $E_\mu = C_\mu^2 = 1$.

\bigskip
\noindent
Christian Seifert and Hendrik Vogt\\
Institut f\"ur Mathematik \\
Technische Universit\"at Hamburg-Harburg\\
21073 Hamburg, Germany \\
{\tt christian.se\rlap{\textcolor{white}{hugo@egon}}ifert@tu-\rlap{\textcolor{white}{darmstadt}}harburg.de}\\
{\tt hendrik.vo\rlap{\textcolor{white}{hugo@egon}}gt@tu-\rlap{\textcolor{white}{darmstadt}}harburg.de}

\end{document}